\documentclass[12pt]{article}
\usepackage{graphicx}
\usepackage{amsfonts}
\usepackage{mathrsfs}
\usepackage[english]{babel}
\usepackage{amsmath, amsthm, amssymb, amsbsy}
\usepackage{mathtools}

\usepackage{enumitem}

\DeclareMathOperator*{\dist}{dist}

\newcommand{\teb}{\mathrm{b}}
\newcommand{\tf}{\tilde{f}}

\newcommand{\bareta}{\bar{\eta}}
\newcommand{\basis}{\mathscr{B}}
\newcommand{\rade}{\mathrm{rad,e}}
\newcommand{\rad}{\mathrm{rad}}
\newcommand{\order}{\mathcal{O}}
\newcommand{\torus}{\mathbb{T}}

\newcounter{thmcounter}
\numberwithin{equation}{section}
\numberwithin{thmcounter}{section}
\theoremstyle{definition}

\theoremstyle{plain}
\newtheorem{lemma}[thmcounter]{Lemma}
\newtheorem{thm}[thmcounter]{Theorem}

\newtheorem{prop}[thmcounter]{Proposition}
\theoremstyle{remark}

\theoremstyle{definition}
\newtheorem{rmkn}[thmcounter]{Remark}

\newcommand{\RS}{\R^{N-1}\times S}

\newcommand{\R}{{\mathbb R}}
\newcommand{\N}{{\mathbb N}}
\newcommand{\Z}{{\mathbb Z}}

\newcommand{\T}{{\mathbb T}}
 \newcommand{\cI}{{\mathcal I}}

\newcommand{\al}{\alpha}
\newcommand{\be}{\beta}

\newcommand{\de}{\delta}
\newcommand{\ep}{\epsilon}

\newcommand{\De}{\Delta}

\newcommand{\om}{\omega}

\newcommand{\Om}{\Omega}

\newcommand{\la}{\lambda}

\newcommand{\ka}{\kappa}
\newcommand{\rest}{\big\arrowvert}

\newcommand{\downto}{\searrow}

\title{The existence of partially localized periodic-quasiperiodic
  solutions and related KAM-type results for elliptic equations on the
  entire space}  
\author{Peter Pol\'{a}\v{c}ik\footnote{Supported in part by the NSF
    Grant DMS--1856491} \\
{\small School of Mathematics}\\
{\small University of Minnesota}\\
{\small Minneapolis, MN 55455}\\
{}\\
{\small and}\\  
 \\
Dar\'{i}o A.\
  Valdebenito\\
{\small Department of Mathematics}\\
{\small University of Tennessee, Knoxville}\\
{\small Knoxville, TN 37996}\\
}
\begin{document}
\maketitle

\begin{abstract} 
We consider the equation
\begin{equation}\tag{1}
\Delta_x u+u_{yy}+f(u)=0,\quad x=(x_1,\dots,x_N)\in\R^N,\  y\in \R,
\end{equation}
where $N\geq 2$ and $f$ is a sufficiently
smooth function satisfying $f(0)=0$,
$f'(0)<0$, and some natural additional conditions.
We prove that equation (1) possesses uncountably many positive solutions
(disregarding translations) which are
radially symmetric in $x'=(x_1,\dots,x_{N-1})$
and decaying as $|x'|\to\infty$, periodic in $x_N$, and
quasiperiodic in $y$. Related theorems for more general
equations are included in our analysis as well. 
Our method is based on center  manifold and KAM-type results. 
\end{abstract}


{\emph{AMS Classification:}   35J61, 35B08, 35B09, 35B10, 35B15. 

\tableofcontents
\section{Introduction}\label{sec:intro}

We consider the semilinear elliptic equation
\begin{equation}\label{eq:1}
\Delta u+u_{yy}+f(u)=0,\quad (x,y)\in\R^N\times \R,
\end{equation}
where  $N\geq 2$ and $f:\R\to\R$ is a $C^k$ function, $k\ge 1$, satisfying
\begin{equation}\label{eq:fcond}
f(0)=0,\ f'(0)<0.
\end{equation}
We generally use the symbol  $\Delta$ for the Laplace operator in the
variables $x=(x_1,\dots,x_N)$, sometimes, when indicated,
only with respect to some of these variables.
We are particularly interested in the more specific equation
\begin{equation}\label{eq:1p}
\Delta u+u_{yy}-u+u^p=0,\quad (x,y)\in\R^N\times \R,
\end{equation}
with  $p>1$.

Equations of the above form, frequently referred to as nonlinear
scalar field equations, have been extensively studied from several
points of view. 
Nonnegative solutions, which we focus on in this paper, are
often the only meaningful solutions from the modeling
viewpoint---thinking of population densities, for example---and also  
they are the only relevant solutions, playing the role of
steady  states, in the dynamics
of the nonlinear heat equation  $u_t=\Delta u+u_{yy}+f(u)$  
with positive initial data. In other applications---for example,
solitary waves or stationary states of nonlinear Klein-Gordon and
Schr\"odinger equations \cite{Berestycki-L:existence}---finite energy
solutions are more relevant.   

Best understood among positive solutions of \eqref{eq:1} are the 
solutions which are (fully) localized in the sense that
they decay to $0$ in all variables $x, y$.
A classical result of  \cite{Gidas-N-N:unbd}  says that 
such solutions are radially symmetric and radially decreasing
with respect to some center in $\R^{N+1}$.
For a large class of nonlinearities, including the
nonlinearity in \eqref{eq:1p}, it is also known that 
the localized positive solution is unique, up to translations, see
\cite{Chen-L:uni,Cortazar-E-F,Kwong,Kwong-L,Peletier-S:uniq1,Serrin-T}. 
For general  results  on the existence and nonexistence of  localized
positive solutions of \eqref{eq:1} we refer the reader to
\cite{Berestycki-L:existence}. We note that, by Pohozaev's identity,
equation \eqref{eq:1p} belongs to the existence class if and only if
$p<(N+3)/(N-1)$ \cite{Berestycki-L:existence, Pohozaev}.

If no decay constraints are imposed, a  variety 
of positive solutions with rather complex structure
is known to exist, including saddle-shaped and     
multiple-end solutions \cite{Cabre-T:saddle,
  Dang-F-P:saddle, delPino-K-P-W2,
  delPino-K-P-W, Kowalczyk-L-P-W} or solutions with 
infinitely many bumps and/or  fronts
(transitions) formed along some directions
\cite{Malchiodi:new, Santra-W}.
Such a diverse set of solutions is hardly amenable to any general
classification or description. 
One then naturally tries to understand various smaller classes of
solutions characterized by some specific symmetry, periodicity, or decay
properties. Similarly as in our previous work,
\cite{p-Valdebenito:hom}, in the present paper we are concerned with
solutions with some predetermined structure with respect to the variables 
$x=(x_1,\dots,x_N)$, that is, all but one variable $y$.
One can think of solutions which are periodic in $x_1,\dots,x_N$,
localized in $x_1,\dots,x_N$, or a combination of these two
structures.
The basic question then is: What can be said about the behavior of
such solutions in the remaining variable $y$?

There is vast literature
on solutions which are periodic in all $x$-variables
and in the remaining  variable $y$ they exhibit 
one or multiple  homoclinic or
 heteroclinic transitions between periodic solutions (see
 \cite{Montecchiari-R, Rabinowitz:reversible} and references therein;
 for related studies of solutions with symmetries instead of the periodicity
 in the $x$ variables see \cite{Alessio-M:mutliplicity} and references therein).

 There is also a number of results concerning
 positive solutions $u$ localized in all of the $x$-variables:
 \begin{equation}
  \label{eq:dec}
  \lim_{|x|\to\infty}\sup_{y\in\R}u(x,y)=0.
\end{equation}
Any such solution is likely radially symmetric
in $x$ about some center in $\R^N$, cp.~\cite{Busca-F, Farina-M-R,
  Gui-M-X},  although this  
has not been proved in the full generality yet. As for the behavior in
$y$, solutions that are periodic (and nonconstant) in $y$ were first found in 
\cite{Dancer:new} and later, by different methods, in
\cite{Alessio-M:energy,Malchiodi:new}. This has been done for a
large class of nonlinearities $f$, including $f(u)=-u+u^p$
with suitable $p>1$. (There is much more to the results in 
\cite{Alessio-M:energy,Dancer:new, Malchiodi:new} than the existence
of periodic solutions; for example, certain global branches of such
solutions where found in \cite{Alessio-M:energy,Dancer:new}). 
In \cite{p-Valdebenito:hom}, we addressed the question whether
positive solutions which are quasiperiodic (and not periodic) in $y$
and satisfy \eqref{eq:dec} exist. We proved that this is indeed the
case if $N\ge 2$ and the nonlinearity $f$ is chosen suitably. 
For a reason that we explain
below, the method used in \cite{p-Valdebenito:hom}
is not applicable in some important specific  equations,
such as \eqref{eq:1p}. The existence of $y$-quasiperiodic solutions satisfying
\eqref{eq:dec} 
for such equations is an open problem which we 
find very interesting, but will not address here.

The structure of solutions that we  examine
in this paper is ``midway'' between full
periodicity and full decay in $x$: the solutions are periodic
 in some of the $x$-variables
and  decay in all the others (this is why we need to assume $N\ge 2$).
For definiteness and simplicity of the exposition,
we specifically postulate the following condition on $u$: writing
$x'=(x_1,\dots,x_{N-1})$, 
 \begin{equation}
  \label{eq:decper}
  \lim_{|x'|\to\infty}\sup_{x_N,y\in\R}u(x',x_N,y)=0,\quad u\text{ is
    periodic in $x_N$,}
\end{equation}
that is, there is  just 1
periodicity variable. Other splits between the
decay and periodicity variables can be treated by our method
in a similar way.

We are mainly concerned with the existence of positive solutions
satisfying \eqref{eq:decper}
which are quasiperiodic in  $y$.  We prove
the existence of such solutions for a  fairly general class of equations.
Our conditions on $f$
require, in addition to \eqref{eq:fcond} and sufficient smoothness,
that the $(N-1)$-dimensional problem
  \begin{equation}\label{eq:gsb}
\Delta u+f(u)=0,\quad x'\in\R^{N-1},
\end{equation}
possesses a {ground state}
which is nondegenerate and has Morse index 1. Let us recall the meaning of
these concepts. By a \emph{ground state} of \eqref{eq:gsb}
we mean a  positive fully
localized solution of \eqref{eq:gsb}. 
From \cite{Gidas-N-N:unbd} we know that any ground state
$u^*$ of \eqref{eq:gsb} is radially symmetric, possibly after a shift
in $\R^{N-1}$, so  we can write $u^*=u^*(r)$, $r=|x'|$.
Consider now the Schr\"{o}dinger operator
$A(u^*)=-\Delta -f'(u^*(r))$, viewed as a self-adjoint operator
 on $L^2_\rad(\R^{N-1})$, the space consisting of all radial
 $L^2(\R^{N-1})$-functions. Its  domain  is 
 $H^2(\R^{N-1})\cap L^2_\rad(\R^{N-1})$.
 Since the potential $f'(u^*(r))$ has the limit $f'(u^*(\infty))=f'(0)<0$,
  the essential spectrum of $A(u^*)$
is contained in  $[-f'(0),\infty)$ (cp.~\cite{Reed-S:IV}).
So
the condition $f'(0)<0$ 
implies that the spectrum in $(-\infty,0]$ 
consists of a finite number of isolated eigenvalues; these eigenvalues
are all simple due to the radial 
symmetry. We say that the ground
state $u^*$ is \emph{nondegenerate} if $0$ is not an eigenvalue of $A(u^*)$.
The \emph{Morse index} of
$u^*$ is defined as the number of negative eigenvalues of $A(u^*)$. 
By a well known instability result,
the Morse index of any ground state is always at least one.

The two conditions, the nondegeneracy and the Morse index equal to
$1$, are  usually satisfied in equations which
have a unique ground state, up to translations (see
\cite{Chen-L:uni,Cortazar-E-F,Kwong,Kwong-L,Peletier-S:uniq1,Serrin-T}).
A typical example
is equation \eqref{eq:gsb} with $f(u)=-u+u^p$ if
$p>1$ is Sobolev-subcritical in dimension $N-1$:
\begin{equation*}
  p<(N+1)/(N-3)_+=
    \begin{cases}
     \  (N+1)/(N-3) & \text{if $N>3$},\\
      \ \infty & \text{if $N\in\{2, 3\}$}.
    \end{cases}
\end{equation*} 
The subcriticality condition is
necessary and sufficient for the existence of a ground state of
\eqref{eq:gsb}, see  
\cite{Berestycki-L:existence}. The uniqueness and the other stated
properties of the ground state are proved in \cite{Kwong}.
Thus our result applies to equation
\eqref{eq:1p} in the subcritical case whenever $f(u)=-u+u^p$
meets our regularity requirement,
which is the case if $p$ is an
integer or if it is large enough.
If $N=2$, the ground state of the one-dimensional problem
\eqref{eq:gsb} is nondegenerate, if it exists, and has Morse
index 1 for any $f$ satisfying \eqref{eq:fcond}.
For $N>2$ and general nonlinearities satisfying \eqref{eq:fcond},
if  ground states on $\R^{N-1}$ exist,
it is not necessarily true that all of them have 
Morse index 1 (see
\cite{Dancer:uniqueness-sing,Davila-dP-G,P:morse}).
However,
under rather general conditions on $f$,
one can find a ground state with this property
as a mountain-pass critical point of the associated energy functional (see
\cite{Dancer:uniqueness-sing,Jeanjean-T}). The nondegeneracy condition
is not guaranteed in general either,
but it is not difficult to show that it holds
``generically'' with respect to $f$ (cp.~\cite[Section
4]{Dancer:new}).

As in \cite{p-Valdebenito:hom},
our method of proving the existence of quasiperiodic solutions
has its grounding in our earlier work
\cite{p-Valdebenito,p-Valdebenito:quadratic}.
It builds on spatial dynamics and  
center manifold techniques for elliptic equations (see
\cite{Kirchgassner}
for the origins of this method,
and, for example, \cite{Calsina-M-SM, DlLlave:center,
  Fiedler-S,Fiedler-Sc:survey,Groves-W, 
Haragus-I:bk,
Mielke:bk,Mielke:grad,Peterhof-S-S,P:common,Vander-I}  
and references therein for further developments) and  
KAM-type results   in a finite-differentiability 
setting.  We remark that related results can be found in
\cite{Scheurle:strip, Valls:quasi}, where quasiperiodic solutions
for elliptic equations on the strip in $\R^2$ have been found.
The center manifold techniques allow us to relate a class of
solutions of the elliptic problem to solutions of a
finite-dimensional Hamiltonian system,
where the variable $y$ plays the role of time.
This is an important step before an application of KAM results, as
the original elliptic equation itself is not a well-posed
evolution problem when  $y$ is viewed as time. Different
approaches to  partial differential
equations which are ill-posed, from the KAM perspective, can be found
in \cite{DlLlave-Sire:KAM, Scheurle:strip}.

In general terms, 
our method consists in the following.   
We consider equations of the form 
\begin{equation}\label{eq:main1}
\Delta u+u_{yy}+a(x)u+f_1(x,u)=0,\quad(x,y)\in
\R^N\times\R=\R^{N+1}, 
\end{equation}
where $f_1(x,u)=u^2g(x,u)$ and all the listed functions are
sufficiently smooth.
The Schr\"odinger operator $-(\De+a(x))$ considered on a suitable
space of functions of $x\in\R^N$---the space reflects the structure of
the solutions one looks for, cp.~\eqref{eq:dec} or \eqref{eq:decper}---is
assumed to have $n\ge 2$ negative eigenvalues, all simple,
with the rest of its spectrum located in the positive half-line.
An application of the center-manifold theorem shows that 
equation \eqref{eq:main1} admits a class of solutions comprising a
finite dimensional manifold. These solutions are in one-to-one
correspondence with solutions of an
ordinary differential equation (ODE) on 
$\R^{2n}$, the \emph{reduced equation},
in which the variable $y$ plays the role of time. 
The reduced equation has a Hamiltonian structure 
and after a sequence of transformations---a Darboux transformation, a
normal form procedure, and action-angle variables---it can be written
in a neighborhood of the origin as a small
perturbation of  an integrable Hamiltonian system. 
The main issue in applying a suitable KAM theorem is
then the verification of a nondegeneracy condition for the integrable
Hamiltonian system.

In \cite{p-Valdebenito:hom}, where we examined solutions localized in
all $x$-variables, we proved that for
suitable nonlinearities $f=f(u)$ all the above requirements are
satisfied by the functions  $a(x)=f'(\varphi(x))$,
$f_1(x,u)=f(\varphi(x)+u)-a(x)u$, where $\varphi$
is a ground state of the
equation
\begin{equation}\label{eq:gs}
\Delta u+f(u)=0,\quad x\in\R^N.
\end{equation}
This way we have proved the
existence of positive $y$-quasiperiodic solutions of \eqref{eq:1}
satisfying \eqref{eq:dec}. Now, when $a(x)$ in
\eqref{eq:main1} is obtained  by the linearization at
the ground state, the assumption  that the operator
$-(\De+a(x))$ on  $L^2(\R^N)$ has  two negative eigenvalues is 
of utmost importance. Equivalently stated, the assumption requires
the ground state $\varphi$ to have Morse index greater than 1.
As mentioned above in connection with the $(N-1)$-dimensional problem
\eqref{eq:gsb}, for many nonlinearities, including   $f(u)=u^p-u$,
it is known that no such ground state can exist. 
Examples of nonlinearities $f$ for which a ground state of \eqref{eq:gs}  
has Morse index greater than 1 do exist, however (see
\cite{Dancer:uniqueness-sing,Davila-dP-G,P:morse}), and to some of
those the results of \cite{p-Valdebenito:hom} apply. 

In our present quest, seeking $y$-quasiperiodic solutions satisfying
\eqref{eq:decper}, we choose $a(x)=f'(\varphi(x))$
as the linearization at a ground state 
$\varphi$ of the equation $\De u+f(u)=0$ in $\R^{N-1}$, rather than $\R^N$.
Viewing $\varphi$ as a function on $\R^N$ constant in 
$x_N$, we consider the operator $-(\De +a(x))$
on a suitable space of functions periodic in $x_N$.
In this setting, it is  relatively easy,  even for $f(u)=u^p-u$,
to arrange that $-(\De +a(x))$ has two
negative eigenvalues by means of a  suitable scaling. Applying then
the general scheme described 
above, we obtain a Hamiltonian reduced equation 
in a form suitable for an application of theorems
from the KAM theory. Here we quickly run into a difficulty, and a major
difference from  \cite{p-Valdebenito:hom}:
the integrable part of this Hamiltonian is necessarily degenerate.
This is due to the symmetries in the problem, regardless of the choice
of the  nonlinearity $f=f(u)$.  
To deal with this difficulty, we use KAM type results
for Hamiltonian systems with ``external parameters'' 
as given in \cite{Broer-H-T, Huitema:thesis}.
It turns out that a scaling parameter which we introduce in
\eqref{eq:1} and which plays the role of an external parameter in the
reduced Hamiltonian gives us enough control over the linear
part of the  Hamiltonian for the KAM type
results to apply.

We formulate our main result, Theorem \ref{thm:main}, on the existence of
$y$-quasiperiodic solutions satisfying \eqref{eq:decper}
in the next section. In the same section, we also state two other new
theorems, Theorem \ref{thm:new} and \ref{thm:newloc},
concerning elliptic equations with parameters. Section \ref{proof1}
contains the proof of 
Theorem \ref{thm:new}, which after minor modifications also gives
the proof of Theorem \ref{thm:newloc}.
We will later show how  \eqref{eq:1} can be put in the context of
such equations by introducing a scaling parameter and thus derive
Theorem \ref{thm:main} from Theorem \ref{thm:new} (see Section
\ref{proof2}).

\section{Statement of the main results} \label{sec:mainres}
In this section, we  first introduce some terminology and notation,
then state our main
results.

Given integers $n\ge 2$, $k\ge 1$, a vector
$\omega=(\omega_1,\dots,\omega_n)\in\R^n$ is said to be
\emph{nonresonant up to order $k$} if 
\begin{equation}
\label{eq:omeganonres} 
\omega\cdot\alpha\neq 0\text{ for all
}\alpha\in\Z^n\setminus\{0\}\text{ such that }\ |\alpha|\leq k.  
\end{equation}
Here $|\alpha|=|\alpha_1|+\dots+|\alpha_n|$, and $\omega\cdot\alpha$
is the usual dot product. 
If \eqref{eq:omeganonres} holds for all $k=1,2,\dots$, we say that
$\om$ is \emph{nonresonant}, or, equivalently, that the numbers
$\omega_1,\dots,\omega_n$ are \emph{rationally independent}. 

A function $u:(x,y)\mapsto u(x,y):\R^N\times\R\to\R$ is said to be
\emph{quasiperiodic} in $y$ if there exist an integer $n\geq 2$, a
nonresonant vector $\omega^*=(\omega^*_1,\dots,\omega_n^*)\in\R^n$,
and an injective function $U$ defined on $\torus^n$ (the
$n$-dimensional torus) with values in the space of real-valued
functions on $\R^N$ such that 
\begin{equation}\label{eq:quasiperiodicdef}
u(x,y)=U(\omega_1^* y,\dots,\omega_n^*y)(x) \quad(x\in\R^N,\,y\in\R).
\end{equation}
The  vector $\omega^*$ is called a \emph{frequency vector} and
its components the \emph{frequencies} of $u$.   
Obviously, there are always countably many frequency vectors of a given
quasiperiodic function, and translations (in $x$ or in $y$) of
quasiperiodic functions are 
quasiperiodic with the same frequencies. 

We emphasize that the nonresonance of the frequency vector 
is a part of our definition. In particular, a quasiperiodic 
function is not periodic and, if it has some regularity properties, 
its image is dense in an $n$-dimensional manifold diffeomorphic 
to $\torus^n$.

We formulate the following hypotheses on the function $f:\R\to\R$.
\begin{itemize}[align=left]
\item[\bf (S)]  $f\in C^{\ell}(\R)$, for some integer $\ell>14+N/2$, and
  $f(0)=0>f'(0)$.
\item[\bf (G)] Equation \eqref{eq:gsb} 
has a nondegenerate ground state $\varphi$ of Morse index 1. 
\end{itemize}

It is well known that  the decay of 
$\varphi$ to zero as $|x'|\to \infty$ is exponential and
$\varphi$ is radial about some center in $\R^{N-1}$ (see
\cite{Gidas-N-N:unbd}). Choosing a suitable
translation, we will always assume that it is
radially symmetric about the origin. 
We will often view $\varphi$ as a function of $x\in\R^N$
independent of the last variable $x_N$.

Our main result reads as follows. 
\begin{thm}\label{thm:main}
Assume that $N\ge 2$ and  \emph{(S), (G)} hold.  
Then there exists an uncountable family of
positive solutions of equation \eqref{eq:1} satisfying
\eqref{eq:decper} such that each of these solutions is
radially symmetric in $x'$, even in $x_N$, and  quasiperiodic in $y$
with  two (rationally independent) frequencies.
The frequency vectors of these quasiperiodic solutions
form an uncountable set in $\R^2$.
\end{thm}
\begin{rmkn}
  \begin{itemize}[align=left,itemindent=3ex,leftmargin=0pt]
   \item[(i)]
  Our proof shows that the family of  solutions as in Theorem
  \ref{thm:main} can be found 
  in any given uniform neighborhood of $\varphi$;
  see Remark \ref{rmk:thmnew}(iii) below. 
 Note, however,
that we cannot guarantee that all these solutions have the same
period in  $x_N$; see Remark \ref{rmk:thmnew}(ii) for an explanation of
this.
 \item[(ii)] As mentioned in the introduction, our theorem applies to equation
\eqref{eq:1p} if $p<(N+1)/(N-3)_+$ is an integer or is sufficiently
large. Specifically, if $p$ is not an integer,
for hypothesis (S) to be satisfied it is sufficient that
$p>14+N/2$.  Note that exponents $p$ satisfying both relations
$14+N/2<p<(N+1)/(N-3)_+$ exist only if $N\le 3$. Integers $p>1$
satisfying $p<(N+1)/(N-3)_+$ exist if $N\le 6$. We remark that
the smoothness in (S) is just a technical, and by no means optimal,
requirement. 
  \end{itemize}
\end{rmkn}

Although the values of $f(u)$ for $u<0$
are irrelevant for the statement of Theorem \ref{thm:main},
it will be convenient to assume   that  
\begin{equation}\label{eq:extendf}
f(u)>0\quad (u<0).
\end{equation}
In view of the conditions $f(0)=0>f'(0)$, this can be arranged,
without affecting the smoothness of $f$, by
modifying $f$ in $(-\infty,0)$.

We will show that Theorem \ref{thm:main} is a consequence of
a more general theorem dealing with the equation
depending on a parameter $s\in \R^d$, $s\approx 0$: 
\begin{equation}\label{eq:a1}
\Delta u+u_{yy}+a(x;s)u+
f_1(x,u;s)=0,\quad x\in \R^N,\ y\in\R.
\end{equation}
Here $f_1$ is a nonlinearity satisfying 
\begin{equation}\label{eq:f}
  f_1(x,0;s)=\frac\partial{\partial u}
  f_1(x,u;s)\rest_{u=0}=0\quad(x\in\R^N,\ s\approx 0),
\end{equation}
and the  functions $a$, $f_1$
are assumed to be  radially symmetric in $x'$, and
even and $2\pi$-periodic in $x_N$. To
indicate the $2\pi$-periodicity in $x_N$, we usually consider
$a$, $f_1(\cdot,u)$  as functions on
$\R^{N-1}\times S$, with $S=\R\mod
2\pi$. We formulate the precise
hypotheses on $a$,  $g$
shortly, after introducing some notation.

We denote by $ C_\teb(\R^N)$ the space of all
continuous bounded  (real-valued) functions on $\R^N$ and by
$ C_\teb^k(\R^N)$  the space of  functions on $\R^N$ with
continuous bounded derivatives up to order $k$,
$k\in\N:=\{0,1,2,\dots\}$. The spaces $ C_{\rade}(\R^{N-1}\times S)$ and
$ C^k_{\rade}(\R^{N-1}\times S)$
are the subspaces of $ C_\teb(\R^N)$ and
$ C_\teb^k(\R^N)$, respectively, consisting of the
functions which are radially symmetric in $x'$, and
 $2\pi$-periodic and even in $x_N$. 
For $k\in\N$, the spaces $L^2_\rade(\RS)$ and
$H^k_\rade(\RS)$ are the closed
subspaces of $L^2(\RS)$ and $H^k(\RS)$,
respectively, consisting of all functions which
are radially symmetric in $x'$ and even in $x_N$.
We assume the standard norms on (the real spaces) $L^2(\RS)$ and
$H^k(\RS)$---for 
example, for $v\in L^2(\RS)$, $\|v\|^2$ is the integral of $v^2$ over 
$\R^{N-1}\times(-\pi,\pi)$---and take the induced norms on the 
subspaces.  

Fix integers $n>1$ (for the number of frequencies of
quasiperiodicity) and $d\ge n-1$ (for the dimension of the parameter
space), and let $B$ be an open neighborhood of the origin in $\R^d$. 
We assume that  the functions 
$a$ and $g$ satisfy
the following hypotheses with some integers 
\begin{equation}\label{eq:Kn}
K> 4n+1,\quad m>\frac{N}2.
\end{equation}

\begin{itemize}[align=left]
\item[\bf (S1)] $a(\cdot;s)\in C^{m+1}_\rade(\RS)$ for each
  $s\in B$, and the map $s\in B\mapsto 
a(\cdot;s)\in C^{m+1}_\rade(\RS)$ is of class 
$ C^{K+1}$.
\item[\bf (S2)]
  $f_1\in C^{K+m+4}(\RS\times \R\times
  B)$, and for all $\vartheta>0$ the function $f_1$ is bounded
  on $\RS\times [-\vartheta,\vartheta]\times B$ together with
  all its partial derivatives up to order $K+m+4$. Also, 
\eqref{eq:f} holds and $f_1(x,u;s)$
  is radially symmetric in $x'$ and even in $x_N$.  
\end{itemize}

The next hypotheses concern the Schr\"{o}dinger operator 
$A_1(s):=-\Delta-a(x;s)$ acting on $L^2_\rade(\RS)$ 
with domain $H^2_\rade(\RS)$. 
\begin{itemize}[align=left]
\item[\bf (A1)(a)] There exists $L<0$ such that
  \begin{equation*}
    \text{
  $\limsup_{|x'|\to\infty} a(x',x_N;s)\leq L$,  uniformly in $x_N$, $s$. }
  \end{equation*}
\item[\bf (A1)(b)] For all $s\in B$, $A_1(s)$ has exactly $n$
  nonpositive eigenvalues, 
  $$\mu_1(s)<\mu_2(s)<\dots<\mu_n(s),$$
  all of them simple, and $\mu_n(s)<0$.  
\end{itemize}
Hypotheses (A1)(a) and (A1)(b) will
sometimes be collectively referred to as (A1). 
Hypothesis (A1)(a) guarantees that for all $s$ 
the essential spectrum $\sigma_{ess}(A_1(s))$ is contained in
$[-L,\infty)$ \  \cite{Dancer:new,Reed-S:IV}.
Since $-L>0$, hypothesis (S1) and the simplicity of the eigenvalues 
 in (A1)(b) imply that $\mu_1(s),\dots,\mu_n(s)$ are $ C^{K+1}$ 
 functions of $s$  (see \cite{Kato:bk}).
 This justifies the use of the derivative in our last hypothesis
 (ND). 
Let $\omega(s):=(\omega_1(s),\dots,\omega_n(s))^T$ (so $\om(s)$ is a column
vector), where
\begin{equation}
  \label{eq:10}
  \omega_j(s):=\sqrt{|\mu_j(s)|}, \quad j=1,\dots,n.
\end{equation}

\begin{itemize}[align=left]
\item[\bf (ND)] The 
   $n\times (d+1)$ matrix $\big[\,\nabla \om(0)\ \  \om(0)\, \big]$
  has rank $n$. 
 \end{itemize}

We can now state our theorem concerning \eqref{eq:a1}.

\begin{thm}\label{thm:new}
Suppose that  hypotheses \emph{(S1)}, \emph{(S2)} 
 (with $K$, $m$ as in \eqref{eq:Kn}), \emph{(A1)}, and \emph{(ND)} are
 satisfied. Then there is an uncountable set  $W\subset \R^n$
 consisting of rationally independent vectors, no two of them
 being linearly dependent, such that for every 
 $(\bar\om_1,\dots,\bar\om_n)\in W$ the following holds: equation \eqref{eq:a1}
 has for some $s\in B$ a solution $u$ 
 such that \eqref{eq:decper} holds, and $u(x,y)$ is radially symmetric
in $x'$, even  and $2\pi$-periodic in $x_N$, 
and quasiperiodic in $y$ with frequencies 
$\bar\om_1,\dots,\bar \om_n$.
\end{thm}

\begin{rmkn}\label{rmk:thmnew}
\begin{itemize}[align=left,itemindent=3ex,leftmargin=0pt]
\item[\rm (i)] Similarly as theorems in
  \cite{p-Valdebenito,p-Valdebenito:quadratic},
  Theorem \ref{thm:new} gives sufficient conditions in terms
  of the coefficients and nonlinearities in a given elliptic equation, 
  presently equation \eqref{eq:a1},  for the existence of
  solutions quasiperiodic in $y$ and satisfying required decay and/or
  symmetry conditions in $x$. The conclusions of the results
  in \cite{p-Valdebenito,p-Valdebenito:quadratic} are in some
  sense stronger: they yield uncountably many quasiperiodic solutions
  for \emph{every} value of the parameter in a certain range
  (which may be required to  be small enough). In contrasts,
  Theorem \ref{thm:new} yields quasiperiodic solutions 
  for \emph{some} values of $s\in B$, possibly leaving out a large set
  of other values.  On the other hand, the present theorem
  has a weaker nondegeneracy condition than the theorems in
  \cite{p-Valdebenito,p-Valdebenito:quadratic}. The nondegeneracy
  conditions in \cite{p-Valdebenito,p-Valdebenito:quadratic} involve some
  nonlinear terms (quadratic or cubic) in the equation,
  whereas our present nondegeneracy condition, (ND), is a condition on
  the coefficient $a$ in the linear part of the equation alone.
  This makes  (ND) much easier to use in
  applications.  Indeed, while the
  nondegeneracy conditions involving nonlinear terms 
  are ``generic'' if the class of admissible nonlinearities
  is large enough, their verification in specific equations, such as the
  spatially homogeneous equation \eqref{eq:1},
  presents a substantial technical hurdle
  (cp.~\cite{p-Valdebenito:hom}). 
  The verification of the present condition (ND)  is, in principle, 
  simpler; it amounts to showing that one has   ``good enough'' control
over the  eigenvalues of a linearized problem when parameters are
varied. 

\item[(ii)] When applying Theorem \ref{thm:new}
  in the proof of Theorem \ref{thm:main}, we introduce
  a parameter $s\in\R$ in
\eqref{eq:1}---so \eqref{eq:1} can be viewed in the context of
\eqref{eq:a1}---by scaling of the variables $(x,y)$.
Therefore, the 
$y$-quasiperiodic solutions which we find using Theorem \ref{thm:new}
for some values of  $s$ will in fact yield, after the
inverse rescaling, $y$-quasiperiodic
solutions of the same original equation \eqref{eq:1} and, due to the
properties of the set $W$, the frequencies of these quasiperiodic
solutions will form an uncountable set.  Note, however, that the
rescaling changes the period in $x_N$. This is  why
we are not able to prescribe the period, say $2\pi$, for the
solutions $u$ in Theorem \ref{thm:main}, with a fixed nonlinearity $f$. 

\item[(iii)] The conclusion of Theorem \ref{thm:new} (as
  well as the conclusion of Theorem \ref{thm:newloc} below) remains valid
  if the solutions $u$ are in addition required to be small in the
  sense that for an arbitrarily given $\ep>0$ one has
  $\sup_{(x,y)\in\R^{N+1}} |u(x,y)|<\ep$.  This follows from the
  proof, where the solutions are found on a local center
  manifold of \eqref{eq:a1}. Accordingly, for any $\ep>0$ one can find
  a solution  $u$ as in Theorem \ref{thm:main} with the property that 
  $\sup_{(x',x_N,y)\in \R^{N+1}} |u(x',x_n,y)-\varphi(x')|<\ep$,
  where  $\varphi$ is the ground state as in (G).
  
\item[(iv)] 
  Evenness with respect to $x_N$ can be dropped in
  the assumptions on $a$ and $g$, and in the definition of the
  domain and the target space of the
operator $A_1(s)=-\Delta-a(x;s)$ (and then it has to be dropped in
the conclusion of  Theorem \ref{thm:new}).
Note, however, that if  $a$, $g$ are even---as will be the
case in an application of Theorem \ref{thm:new}
below---the eigenvalues $\mu_2(s),\dots,\mu_n(s)$ of
the operator $-\Delta-a(x;s)$ may be simple in the space of even
functions but not in the full space. Similarly, it is possible to drop
the assumption of radial symmetry in $x'$, but the simplicity of the
eigenvalues may fail to hold in the full space. 

\item[(v)]  A nondegeneracy condition of the same form as (ND)
  appears in Scheurle's paper 
  \cite{Scheurle:ode} on bifurcations of quasiperiodic
  solutions in analytic reversible ODEs. He used techniques similar to
  \cite{Scheurle:ode} in the paper \cite{Scheurle:strip}, already
  mentioned in the introduction, on (analytic)
  elliptic equations on the strip $\{(x,y):x\in(0,1),y\in\R\}$ .
\end{itemize}
\end{rmkn}

The localized-periodic setting in which we consider
equation \eqref{eq:a1} reflects our goal to study solutions 
satisfying \eqref{eq:decper}. However, 
our present techniques can be used 
in other settings; for example, one can consider
a different split between periodicity and decay variables in
$x_1,\dots,x_N$. Straightforward, mostly notational,
modifications of the arguments below apply in any such
 setting. As an illustration, we
formulate a  theorem analogous to Theorem
\ref{thm:new} in but one different setting:
the symmetry and decay (and no periodicity) in all variables
$x$. 

We need the following spaces: 
$ C_{\rad}(\R^{N})$,
$ C^k_{\rad}(\R^{N})$ consist of all 
radially symmetric functions in  $ C_\teb(\R^N)$ and
$ C_\teb^k(\R^N)$, respectively; $L_{\rad}^2(\R^N)$ is the space
of all radial $L^2(\R^N)$-functions,  and for 
$k\in\N$,  $H^k_{\rad}(\R^N):=H^k(\R^N)\cap L^2_\rad(\R^N)$ is the space of all radial
$H^k(\R^N)$-functions.

\begin{thm}\label{thm:newloc}
  Let $K$ and $m$ be as in \eqref{eq:Kn}.
Assume that hypotheses \emph{(S1)}, \emph{(S2)}, \emph{(A1)}, 
 \emph{(ND)}  are satisfied with 
$ C^{m+1}_{\rade}(\R^{N-1}\times S)$ replaced by
$ C^{m+1}_{\rad}(\R^{N})$,
$ C^{K+m+4}(\RS\times \R\times  B)$ by
$ C^{K+m+4}(\R^N\times\R\times
  B)$, $L^2_\rade(\RS)$ by $L_{\rad}^2(\R^N)$, and 
  $H^2_\rade(\RS)$ by $H_{\rad}^2(\R^N)$;
and the last assumption in\emph{ (S2)} (radial symmetry in $x'$ and
periodicity in $x_n$) replaced by the assumption that $f_1$ is radially
symmetric in $x$. Then there is an uncountable set  $W\subset \R^n$
 consisting of rationally independent vectors, no two of them
 being linearly dependent, such that for every 
 $(\bar\om_1,\dots,\bar\om_n)\in W$ the following holds:
 equation \eqref{eq:a1} has for some $s\in B$ a solution $u$ 
 such that \eqref{eq:dec} holds, and $u(x,y)$ is radially symmetric
in $x$
and quasiperiodic in $y$ with frequencies 
$\bar\om_1,\dots,\bar \om_n$.
\end{thm}

For the proof of this theorem, one just needs to make obvious changes
in the proof of Theorem \ref{thm:new}
consisting mostly of replacements of
the underlying spaces as in the formulation of the theorem. 

\begin{rmkn}
If one considers periodicity in two or more variables (say,
$(x_1,\dots,x_j)$), the dependence of $a$ and $f_1$ on those variables
may also impose some additional restrictions on the setting, for
instance, if $a_1$ and $f$ do not depend on $(x_1,\dots,x_j)$, then
the corresponding periods must be chosen suitably to keep the
simplicity of the eigenvalues of $-\Delta-a(x;s)$. 
\end{rmkn}

\section{Proof of Theorem \ref{thm:new}}\label{proof1}

We use the notation introduced in the previous section and assume
hypotheses {(S1)}, {(S2)}, {(A1)}, 
 {(ND)}  to be satisfied.
Let $B_\de:=\{s\in\R^d:|s|<\de\}$, where we take $\de>0$ so that
$B_\de\subset B$ (below we will make $\de>0$ smaller several times). 

For $s\in B_\de$
and $j=1,\dots,n$, we 
denote by $\varphi_j(\cdot;s)$ an  eigenfunction of the operator
$A_1(s)$ associated with the eigenvalue $\mu_j(s)$
normalized in the $L^2$-norm. For 
the \emph{principal eigenfunction}  $\varphi_1(\cdot;s)$, we may
assume that it is positive which determines it uniquely,
and it is then of class $ C^{K+1}$ as a
$H^2_\rade(\RS)$-valued function of $s$
(see \cite{Kato:bk}). The same applies to $\varphi_j(\cdot;s)$, provided
it is chosen suitably (the normalization determines it uniquely up to
a sign). Since $\mu_1(s)<\dots<\mu_n(s)$ are simple isolated eigenvalues
of $A_1(s)$, the eigenfunctions
$\varphi_1(\cdot;s),\dots, \varphi_n(\cdot;s)$ have  
exponential decay as $|x'|\to
\infty$ \cite{Agmon, Reed-S:IV}. 

Since the essential spectrum of $A_1(s)$ is contained in
$[-L,\infty)$, 
the eigenvalues in $(-\infty,-L)$ are isolated in
$\sigma(A_1(s))$ and  hypotheses  (A1)(a), (A1)(b) imply that there is
$\gamma>0$ such that 
$(0,\gamma)\cap\sigma(A_1(s))=\emptyset$ for all $s\in B_\de$.

Hypotheses (S1), (S2), (A1)(a), (A2)(b), (NR) are
  analogous to some hypotheses in our previous papers 
  \cite{p-Valdebenito,p-Valdebenito:quadratic}. In those papers
  we mainly focused on solutions which are radially
  symmetric and decaying in all variables $x$ and,
accordingly, the assumptions on the functions $a$, $f_1$
involved radial symmetry in $x$.  In the present setting, we assume
radial symmetry in $x'$ and 
periodicity in $x_N$. As noted in 
\cite[Remark~2.1(v)]{p-Valdebenito}, 
\cite[Remark~2.1(ii)]{p-Valdebenito:quadratic},
the general technical results from \cite{p-Valdebenito,p-Valdebenito:quadratic}
apply in the present setting
with straightforward modifications of the proofs.
In the next subsection, we recall the needed results from
\cite{p-Valdebenito,p-Valdebenito:quadratic}.

\subsection{Center manifold and the structure of the reduced
  equation}\label{sec:reduced}  

Here we essentially just reproduce Section 3 of
\cite{p-Valdebenito:quadratic} (which in turn is an
extension of results in Sections
3 and 4 of  \cite{p-Valdebenito}) with minor adjustments in the
notation on the account of the present
periodicity-decay setting. The fact that
$s\in B_\delta\subset \R^d$, whereas in
\cite{p-Valdebenito:quadratic} we had $s\in
(-\de,\de)\subset \R$, makes no  nontrivial
difference in the proofs.

We begin with the center manifold reduction. For that we first write
equation \eqref{eq:a1} in an abstract form, using the spaces
$X:=H^{m+1}_\rade(\RS)\times H^{m}_\rade(\RS)$, and
$Z:=H^{m+2}_\rade(\RS)\times 
H^{m+1}_\rade(\RS)$. 
Let $f_1$ be as in \eqref{eq:f}. Its  Nemytskii operator
$\tf:H^{m+2}_\rade(\RS)\times  B_\de\to H^{m+1}_\rade(\RS)$ is given by
\[
\tf(u;s)(x)=f_1(x,u(x);s),
\]
and it  a well defined map of class $ C^{K+1}$
(see \cite[Theorem A.1(b)]{p-Valdebenito}).
The abstract form of \eqref{eq:a1} is 
\begin{equation}\label{eq:abstractsys}
\begin{aligned}
\frac{du_1}{dy}&=u_2,\\
\frac{du_2}{dy}&=A_1(s)u_1-\tf(u_1;s).
\end{aligned}
\end{equation}
We rewrite this further as 
\begin{equation}\label{eq:absys2}
\frac{du}{dy}=A(s)u+R(u;s),
\end{equation}
where $u=(u_1,u_2)$,
\begin{equation}
  \label{eq:8}
  \begin{aligned}
A(s)(u_1,u_2)&=(u_2,A_1(s)u_1)^T,\\
R(u_1,u_2;s)&=(0,\tf(u_1;s))^T.
\end{aligned}
\end{equation}
Here, for each $s\in  B_\de$, $A(s)$ 
is considered as an operator on $X$ with domain
$D(A(s))=Z$, and $R$ as a $ C^{K+1}$-map from 
 $Z\times  B_\de$ to $Z$. 
The notion of a solution of \eqref{eq:absys2} on an interval $\cI$ is 
as in \cite{Haragus-I:bk,Vander-I}: it is a function
in  $ C^1(\cI,X)\cap 
 C(\cI,Z)$ satisfying \eqref{eq:absys2}. 

Recall that $\varphi_j(\cdot;s)$, $j=1,\dots,n,$ are
the eigenfunctions of $A_1(s):=-\Delta-a(x;s)$ corresponding to
the eigenvalues $\mu_1(s),\dots,\mu_n(s)$, and they have been chosen so
that they are of class
$ C^{K+1}$ as $H^2_\rade(\RS)$-valued functions of $s$. By elliptic
regularity, for $j=1,\dots,n$,  
$\varphi_j(\cdot;s)\in H^{m+2}_\rade(\RS)$
and it is of class $ C^{K+1}$ as a  $H^{m+2}_\rade(\RS)$-valued
function of $s$. 
Define the space
\[
X_c(s):=\big\{(h,\tilde{h})^T:h,\tilde{h}\in
\mathrm{span}\{\varphi_1(\cdot;s),\dots,\varphi_n(\cdot;s)\}\big\}\subset
Z,
\]
the orthogonal projection operator $$
\Pi(s):L^{2}_\rade(\RS)\to
\mathrm{span}\{\varphi_1(\cdot;s),\dots,\varphi_n(\cdot;s)\},$$
and let
$P_c(s):X\to X_c(s)$ be given by $P_c(s)(v_1,v_2)=(\Pi(s) v_1,\Pi(s)
v_2)$. As shown in \cite[Section 3.2]{p-Valdebenito},  
$P_c(s)$ is the spectral projection for the operator $A(s)$ associated
with the spectral set $\{\pm i\om_j(s):j=1,\dots,n\}$
(with $\om_j(s)$ as in \eqref{eq:10})---the spectrum
of $A(s)$ is the union of this set and a set which is at a positive
distance from the imaginary axis. 
The smoothness of the maps  $s\mapsto \varphi_j(\cdot;s)$
implies that  $s\mapsto P_c(s)$ is of class $ C^{K+1}$ as an
$\mathscr{L}(X,Z)$-valued map on $B_\de$.

 Also define $P_h(s)=I_X-P_c(s)$, $I_X$ being the identity map on $X$, 
and, for $j=1,\dots,n$,
\begin{equation}\label{eq:psizeta}
\psi_j(\cdot;s)=
(\varphi_j(\cdot;s),0)^T,\quad\zeta_j(\cdot;s)=(0,\varphi_j(\cdot;s))^T.  
\end{equation}
A basis of $X_c(s)$ is given by
\[
  \basis(s):=
  \{\psi_1(\cdot;s),\dots,\psi_n(\cdot;s),\zeta_1(\cdot;s),\dots,\zeta_n(\cdot;s)\}. 
\]
For $z\in X_c(s)$, we denote by $\{z\}_{\basis}$ the coordinates of
$z$ with respect to the basis $\basis(s)$. Denote further
\begin{equation}\label{eq:psizetav}
\begin{aligned}
\psi(s)&:=(\psi_1(\cdot;s),\dots,\psi_n(\cdot;s)),\\
\zeta(s)&:=(\zeta_1(\cdot;s),\dots,\zeta_n(\cdot;s)).
\end{aligned} 
\end{equation}

The following result is a part of
\cite[Proposition~3.1]{p-Valdebenito:quadratic}, 
adjusted to the present setting.  
\begin{prop}\label{prop:wc} 
Using the above notation, the following statement is valid, possibly
after making $\de>0$ smaller.  
There exist a map $\sigma:(\xi,\eta;s)\in\R^{2n}\times B_\de\mapsto
\sigma(\xi,\eta;s)\in Z$ of class $ C^{K+1}$ and a neighborhood
$\mathscr{N}$ of $0$ in $Z$ such that for each $s\in B_\de$ one has
\begin{align}
\label{eq:sigma-prop1}
\sigma(\xi,\eta;s)\in P_h(s)Z\quad ((\xi,\eta)\in\R^{2n}),\\
\label{eq:sigma-prop2}
\sigma(0,0;s)=0,\quad D_{(\xi,\eta)}\sigma(0,0;s)=0,
\end{align} 
and  the manifold
\[
W_c(s)=\{\xi\cdot\psi(s)+\eta\cdot\zeta(s)+
\sigma(\xi,\eta;s):(\xi,\eta)=
(\xi_1,\dots,\xi_n,\eta_1,\dots,\eta_n)\in\R^{2n}\}\subset Z  
\]
has the following properties:
	\begin{itemize}
	\item[(a)] If $u(y)$ is a solution of \eqref{eq:abstractsys} on $\cI=\R$
          and $u(y)\in \mathscr{N}$ for all $y\in\R$, then $u(y)\in
          W_c(s)$ for all $y\in\R$; that is, $W_c(s)$
          contains the trajectory of each solution of \eqref{eq:abstractsys}
          which stays in  $\mathscr{N}$ for all $y\in \R$.  
	
	\item[(b)] If $z:\R\to X_c(s)$ is a solution of the equation
	\begin{equation}
	\label{eq:reduced2} \frac{dz}{dy}=A(s)\big|_{X_c(s)}z+P_c(s)
        R(z+\sigma(\{z\}_\basis;s);s) 
	\end{equation}
	on some interval $\cI$, and 
        $u(y):=z(y)+\sigma(\{z(y)\}_\basis;s)\in\mathscr{N}$ for all $y\in \cI$,
        then $u:\cI\to Z$ is a solution of \eqref{eq:abstractsys} on $\cI$. 
	\end{itemize}
\end{prop}
In the sequel, $W_c(s)$ is called the \emph{center
  manifold} and equation \eqref{eq:reduced2}  the \emph{reduced
  equation}.

\vspace{8pt}

Next, we examine  the Hamiltonian structure of the reduced equation. 
For $(u,v)\in Z$ and any fixed $s\in B_\de$, let
\begin{equation}\label{eq:originalHam}
H(u,v)=\int_{\RS}\left(\frac{-1}{2}|\nabla
u(x)|^2+\frac{1}{2}a(x;s)u^2(x)+F(x,u(x);s)+\frac{1}{2}v^2(x)\right)dx,
\end{equation}
where
\[
F(x,u;s)=\int_0^u f_1(x,\vartheta;s)d\vartheta.
\]
Equation \eqref{eq:abstractsys} has a formal Hamiltonian structure
with respect to the functional $H$ and  this
structure is inherited in a certain way by the reduced equation. 
More specifically, denoting by $\Phi$ the composition of the maps 
$(\xi,\eta)\to \sigma(\xi,\eta;s):\R^{2n}\to Z$ and $H:Z\to \R$,  
\eqref{eq:reduced2} is the Hamiltonian 
system with respect to the Hamiltonian $\Phi$ 
and a certain symplectic structure defined  
in a neighborhood of $(0,0)\in\R^{2n}$. 
This is a consequence of general results
of \cite{Mielke:bk}; in 
\cite{p-Valdebenito} we  gave 
a proof, with some additional useful 
information, using direct explicit computations. 
We have then  transformed the system by
performing several coordinate changes.
By the first one, we achieve that, near the origin,  in the new coordinates
$(\xi',\eta')$ the system is Hamiltonian with respect to
(the transformed Hamiltonian) and the standard symplectic form on
$\R^{2n}$, $\sum_i\xi'_i\wedge \eta'_i$.
The existence of such a local transformation is guaranteed by 
the Darboux theorem, but in 
\cite{p-Valdebenito}
we took some care to keep track of how the symplectic structure and
the Darboux transformation depend on the parameters. 
We showed in particular that 
the Darboux transformation can be chosen as
a $ C^K$ map in $\xi$, $\eta$, and $s$, which is 
the sum of the identity map on $\R^{2n}$ and terms of order
$\order(|(\xi,\eta)|^3)$.  
In the  coordinates $(\xi',\eta')$ 
resulting from such a transformation, the
Hamiltonian takes the
following form for $(\xi',\eta')\approx (0,0)$:
\begin{equation} 
\label{formPhip}
\Phi(\xi',\eta';s) 
=\frac{1}{2}\sum_{j=1}^n (-\mu_j(s)(\xi'_j)^2+(\eta'_j)^2)
+\Phi'(\xi',\eta';s).
\end{equation}
Here, $\mu_j(s)$ are the negative eigenvalues of $A_1(s)$, as above,
and $\Phi'$ is a function of class $ C^K$ in all its arguments
and of order $\order(|(\xi',\eta')|^3)$ as $(\xi',\eta')\to (0,0)$.
We remark that the formulas given for $\Phi$ in 
\cite{p-Valdebenito,p-Valdebenito:quadratic}
are a bit longer, specifying in particular
the cubic terms of $\Phi$, but those more precise expressions
are not needed here. 

We now make a canonical (that is, symplectic form
preserving) linear transformation defined by
\begin{equation}\label{eq:symplin}
\xi'_j=\frac{1}{\sqrt{\omega_j(s)}}\xi_j,\quad
\eta'_j=\sqrt{\omega_j(s)}\,\eta_j\quad(j=1,\dots,n), 
\end{equation}
where $\omega_j(s):=\sqrt{|\mu_j(s)|}$,  $j=1,\dots,n$, are as in
\eqref{eq:10}. (The coordinates 
$\xi$ and $\eta$ used here are not the same coordinates as in
Proposition \ref{prop:wc}.)
This transformation
puts the quadratic part of $\Phi$ in the ``normal form:'' in the
coordinates $(\xi,\eta)$,
\begin{equation}
  \label{eq:11}
  \Phi(\xi,\eta;s):=\frac{1}{2}\sum_{j=1}^n
\omega_j(s)(\xi_j^2+\eta_j^2)+\hat \Phi(\xi,\eta;s),
\end{equation}
where $\hat \Phi$ is a function of class $ C^K$
and of order $\order(|(\xi,\eta)|^3)$ as $(\xi,\eta)\to (0,0)$.

Later, we will also use
the \emph{action-angle} variables 
$J=(J_1,\dots,J_n)\in \R^n$, 
$\theta=(\theta_1,\dots,\theta_n)\in\torus^n$.
They are defined by
\begin{equation}
  \label{eq:12}
  (\xi_j,\eta_j)=\sqrt{2J_j}(\cos\theta_j,\sin\theta_j)
\end{equation}
in regions where  $J_j=(\bar\xi_j^2+\bareta_j^2)/2>0$
for all $j\in\{1,\dots,n\}$.
In these coordinates, the Hamiltonian  $\Phi$ in \eqref{eq:11} 
takes the form
\begin{equation}\label{eq:PhiThetaI}
  \Phi(\theta,J;s)=\omega(s)\cdot J+\hat\Phi(\theta,J;s)
\end{equation}
(with the usual abuse of notation: $\hat\Phi(\theta,J;s)$ actually stands
for $\Phi(\xi(\theta,J),\eta(\theta,J);s)$). The  change of
coordinates from $(\xi_j,\eta_j)$ to $(\theta,J)$ is also
canonical. In particular, in these coordinates the reduced equation
reads as follows: 
\begin{equation}
  \label{eq:28}
  \begin{aligned}
     \dot \theta&= \nabla_J \Phi(\theta,J;s),\\  
     \dot J&=-\nabla_\theta \Phi(\theta,J;s).
  \end{aligned}
\end{equation}

The above Hamiltonian structure is the structure we use below in the proof
of Theorem \ref{thm:new}. We remark that another structure we could
use instead is the reversibility of \eqref{eq:abstractsys}:
if $(u_1(x,y),u_2(x,y))$ a solution, so is
$(u_1(x,-y),-u_2(x,-y))$). This reversibility structure
is also inherited by the reduced equation (see
\cite{Haragus-I:bk, Mielke:bk}). More specifically, writing the
equation as an ODE on $\R^{2n}$,
there is a transformation $D$ on $\R^{2n}$ such that $D^2$ is the 
identity map on $\R^{2n}$
and $D$ anticommutes with the right-hand side of the ODE.
(See Remark \ref{rm:discussion} for additional comments on the
reversibility structure).

\subsection{KAM-type results for systems with parameters and
  completion of the proof of Theorem \ref{thm:new}}
\label{sec:appl} 

To prove Theorem \ref{thm:new}, we apply 
a  KAM-type result from \cite{Broer-H-T, Huitema:thesis}
to the reduced Hamiltonian \eqref{eq:PhiThetaI}.  
To recall that  result,
consider, for some positive integers $n$ and  $d$,
a Hamiltonian $H:\torus^n\times\Omega\times B\to\R$ given by 
\begin{equation}
\label{eq:Hamdec}H(\theta,I;s)=H^0(I;s)+H^1(\theta,I;s),
\end{equation}
where $\torus^n=\R^n/(2\pi\Z^n)$ is the $n$-dimensional torus
(so $H^1(\theta,I;s)$ is $2\pi$-periodic in
$\theta_1, \dots,\theta_n$), and
$\Om$, $B$ are bounded domains in $\R^n$, $\R^d$, respectively; $s\in
B$ acts as a
parameter. We assume that $H^0$ is (real) analytic on
$\Om\times B$ and $H^1:\torus^n\times\Omega\times B\to\R$ is of class
$C^k$ for some $k\ge 2$. 

The Hamiltonian system corresponding to $H$ is
\begin{equation}
  \label{eq:28par}
  \begin{aligned}
     \dot \theta&= \nabla_I H(\theta,I;s),\\  
     \dot I&=-\nabla_\theta H(\theta,I;s),
  \end{aligned}
\end{equation}
and the one corresponding to  $H^0$,
\begin{equation}
  \label{eq:280}
  \begin{aligned}
     \dot \theta&= \nabla_I H^0(I;s),\\  
     \dot I&=0.
  \end{aligned}
\end{equation}
 We denote by $\om^*$ 
 the \emph{frequency map} of $H^0$:
 \begin{equation}
   \label{eq:9}
 (I;s)\mapsto \omega^*(I;s):=(\nabla_I H^0(I;s))^T: \Omega\times B\to \R^{n}.
\end{equation}
Here and below we view the gradient as a row vector, so
$\omega^*(I;s)$ is a column vector.

For each $s\in B$, the system  \eqref{eq:280} is completely \emph{integrable}.
Its state space is covered by invariant tori 
$\T^n\times \{I_0\}$, $I_0\in \Om$, and any such torus is  filled with 
trajectories of quasiperiodic solutions whenever the vector
$\omega^*(I_0;s)$ is nonresonant. As usual,
for the persistence of some 
of these quasiperiodic tori under the perturbation in
\eqref{eq:Hamdec}, we introduce a class of Diophantine frequencies.  
A vector $\om\in \R^n$ is said to be 
$\kappa,\nu$\emph{-Diophantine},
for some $\kappa>0$ and $\nu>n-1$, if 
\begin{equation}
\label{eq:omegaDioph} |\omega\cdot\alpha|\geq \kappa
|\alpha|^{-\nu}\quad
(\alpha\in \Z^n\setminus\{0\}).
\end{equation}
Fixing $\nu>n-1$ arbitrarily, for any nonempty bounded open
set $V\subset \R^n$ and
$\kappa>0$, we define 
\begin{equation}\label{eq:vkappa}
V_\kappa:=\{\omega\in V:\dist(\omega,\partial V)\geq\kappa\text{ and
}\omega\text{ is }\kappa,\nu\text{-Diophantine}\}. 
\end{equation}
It is well known that for small $\ka>0$ the Lebesgue measure,
$|V_\kappa|$, of $V_\ka$ is positive; in fact,  $|V\setminus V_\kappa|\to 0$ as
$\ka\downto0$.

As a nondegeneracy assumption, we shall require the frequency map
$$\omega^*(I,s)=(\om^*_1(I,s),\dots, \om^*_n(I,s))^T$$
to have surjective derivative:
\begin{itemize}[align=left]
\item[\bf (NDsI)]  The 
  $n\times (n+d)$ matrix 
  \begin{equation*}
  \nabla_{I,s}\om^*(I,s)=  \begin{bmatrix}
    \   \nabla_{I,s}\om^*_1(I,s)\\
      \vdots \\
    \   \nabla_{I,s}\om^*_n(I,s)\ 
    \end{bmatrix}
  \end{equation*}
has rank $n$ for all $(I,s)\in\Om\times B$. 
\end{itemize}
Note that this assumption implies  that the range of $\om^*$,
$V=\om^*(\Om\times B)$, is an open set in $\R^n$.   

The perturbation term $H^1$ will be assumed to
have a sufficiently small norm $ C^k$-norm
$\|H^1\|_{ C^k(\torus^n\times\Omega\times B)}$
which stands for the smallest upper bound, over $\torus^n\times\Omega\times B$,
on the moduli of all  derivatives of $H^1$ of orders 0 through  $k$.

\begin{thm}\label{thm:BHT} Let $H^0$,
  $\om^*$ be as above and
  $V:=\om^*(\Om\times B)$. Assume that \emph{(NDsI)} holds and
  let $\nu>n-1$ be fixed.  If   $k_0=k_0(\nu)$ is a sufficiently large
  integer, then the following statement holds.
  For every $\kappa>0$
  there is $\vartheta>0$ such that for an arbitrary $ C^k$-map
  $H^1:\torus^n\times\Omega\times B\to\R$ with $k\ge k_0$ and
  $\|H^1\|_{ C^k(\torus^n\times\Omega\times B)}<\vartheta$
  the Hamiltonian $H^0+H^1$ has the following property. 
  There is a $C^1$ map
  \begin{equation*}
    \Psi: \torus^n\times\Omega\times B\to \torus^n\times\R^n\times \R^d
  \end{equation*}
 of the form
  \begin{equation}
    \label{eq:2}
    \Psi(\theta,I,s)=(T(\theta,I,s), \Upsilon(I,s)),\quad
    T(\theta,I,s)\in \torus^n\times\R^n, \ \Upsilon(I,s)\in\R^d,
  \end{equation}
    which is a near-identity diffeomorphism onto its image and such that
    for any $(I_0,s_0)\in\torus^n\times\Omega$ with 
  $\om^*(I_0,s_0)\in V_\ka$ the manifold
  \begin{equation}
    \label{eq:3}
    \tilde\torus_{(I_0,s_0)}:=\{T(\theta,I_0,s_0):\theta\in \torus^n\}
  \end{equation}
  is invariant under the flow of \eqref{eq:28par} with $s= \Upsilon(I_0,s_0)$
  and the solution of 
\eqref{eq:28par} with the initial condition 
$T(\theta_0,\om^*(I_0,s_0))$, $\theta_0\in \torus^n$,  is given by 
$T(\theta_0+\om^*(I_0,s_0)t,\om^*(I_0,s_0))$, $t\in\R$. 
\end{thm}
This is a special case of a theorem from \cite{Broer-H-T}: see
Corollary 5.1 and 
Section 5c in \cite{Broer-H-T} for a version of the theorem for
analytic Hamiltonians; the adjustments needed
in the proof for finitely differentiable Hamiltonians are
indicated in the appendix of \cite{Broer-H-T} (see also
\cite{Huitema:thesis}; statements of the theorem and related results
can also be found in  \cite{Broer-H-N, Sevryuk:KAM-stable}). 
The theorem is an extension of a result of
\cite{Poschel:integrability} for a Hamiltonian without parameters
(that is, $d=0$), in which case 
condition (NDsI) is the same as the 
Kolmogorov nondegeneracy condition.

\begin{rmkn}
  \label{rm:BHT}
  \begin{itemize}[align=left,itemindent=3ex,leftmargin=0pt]
  \item[(i)] By saying that $\Psi$ is a near-identity diffeomorphism
    we mean that the $ C^1$ norm of the difference of $\Psi$ and
    the identity on $\torus^n\times\Omega\times B$ is less than 1. One
    can additionally say that the norm becomes arbitrarily small
    as $\vartheta\to0$. 
  \item[(ii)]  Since $V_\ka$ consists of nonresonant vectors, the solution
 $$t\mapsto T(\theta_0+\om^*(I_0,s_0)t,\om^*(I_0,s_0))$$
 is quasiperiodic with
 the frequency vector $\om^*(I_0,s_0)$. The set of the
 frequencies of these solutions,  $V_\ka$, has positive measure if
 $\ka$ is sufficiently small. 
 \item[(iii)]
   Specific estimates as to how large  $k_0=k_0(\tau)$ has to be
   are available.
  As noted in \cite[Appendix]{Broer-H-T}, a sufficient but not
  optimal condition is $k_0>4\nu+2$. Thus, if a regularity class
  $ C^k$ with $k>4n-2$ is given upfront,
  one can always pick $k_0\le k$ and $\nu>n-1$ so
  that $k_0>4\nu+2$ and then Theorem \ref{thm:BHT} applies
  with such choices of $\nu$ and $k_0$. 
  We also remark that the diffeomorphism $\Psi$ is more regular
  than $C^1$ and its smoothness increases with  $k$  
(see \cite{Poschel:integrability} for more precise differentiability 
assumptions on the Hamiltonian and the corresponding 
regularity properties of the map $T$ in the case $d=0$).
  \end{itemize}
\end{rmkn}

In our application of Theorem \ref{thm:BHT}, we  consider 
a Hamiltonian $G:\torus^n\times  \Omega\times
B\to\R$ given by 
\begin{equation}
\label{eq:HamG}G(\theta,I;s)=\om(s)\cdot I+G^1(\theta,I;s),
\end{equation}
where $s\mapsto\om(s):B\to \R^n$ is a $C^1$ map satisfying
the following condition.
\begin{itemize}[align=left]
\item[\bf (NDs)] The
  $n\times (d+1)$ matrix
$$\left[\ \nabla\om(s)\quad \om(s)\ \right]$$
has rank $n$ for all $s\in B$. 
\end{itemize}

Note that this is the type of condition satisfied locally
by the frequencies in our elliptic problem, 
see condition (ND) in Section \ref{sec:mainres}. 

We will take
the \emph{linear} function $G^0(I;s)=\om(s)\cdot I$ as the
unperturbed integrable Hamiltonian and view $G^1$ as a small
$C^k$ perturbation. We relate the Hamiltonians $G$ and
$H$---and conditions (NDs) and (NDsI)---in the following lemma. In the
simplest case, when $s\mapsto\om(s)$ is analytic and $\nabla \om(s)$ alone has
rank $n$, we can simply take $H^0=G^0$. This leads to a very similar
setup, with the frequencies serving as parameters,
as in \cite{Poschel:lecture} where a parametrization by frequencies
is used in the  proof of a classical KAM theorem
(see also  \cite{Moser:series} for an earlier use of a 
``parametrization'' technique). In other cases, some
 ``tricks'' will be used to accommodate $G^0$
in the setting of Theorem \ref{thm:BHT}.

\begin{lemma}
  \label{le:NDs} Fix $\nu>n-1$ and let $k_0=k_0(\nu)$ be
  as in Theorem \ref{thm:BHT}. Given any $k\ge k_0$, assume that
  $s\mapsto\om(s):B\to \R^n$ is a $ C^k$
 map satisfying \emph{(NDs)}. Then there is $\vartheta>0$ such that 
 for an arbitrary $ C^k$-map
  $G^1:\torus^n\times\Omega\times B\to\R$ with
  $\|G^1\|_{ C^k(\torus^n\times\Omega\times B)}<\vartheta$
  the Hamiltonian $G:=G^0+G^1$ has the following property.
  There is an uncountable set  $W\subset \R^n$
 consisting of rationally independent vectors, no two of them
 being linearly dependent, such that for every 
 $\bar\om\in W$ the Hamiltonian system
\begin{equation} \label{eq:17}
  \begin{aligned}
     \dot \theta&= \nabla_I G(\theta,I;s),\\  
     \dot I&=-\nabla_\theta G(\theta,I;s)
  \end{aligned}
\end{equation}
has for some $s\in B$ a quasiperiodic  solution of the form
$T_s(\bar \om t)$, $t\in\R$, where $T_s:\torus^n\to \torus^n\times\Om$ is
a $C^1$ imbedding of the torus $\torus^n$.
\end{lemma}
\begin{proof}
  First  assume that $\nabla \om(s)$ has
  rank $n$ for all $s\in B$ and $s\mapsto\om(s)$ is analytic.
  Taking $H^0(I;s):=G^0(I;s)=\om(s)\cdot I$ for all $I\in\Om$,
$s\in B$, we immediately see that condition (NDsI) is satisfied
with $\om^*(I,s)=\om(s)$ (cp.~\eqref{eq:9}). Let $V$ be the image of
$B$ under the map $s\to\om(s)$. This is an open set in
$\R^n$,  hence for $\kappa>0$ small enough,
the set $V_\kappa$ has positive measure.
Fix such $\ka$ and let
$\vartheta=\vartheta(\kappa)$ be as in Theorem \ref{thm:BHT}.
We claim that the conclusion of Lemma \ref{le:NDs} holds with this
$\vartheta$. Indeed, if $G^1$ satisfies the smallness condition, then
Theorem \ref{thm:BHT} with $H^1=G^1$ tells us that the conclusion
of Lemma \ref{le:NDs} regarding \eqref{eq:17} holds for any
$\bar\om \in V_\ka$: we simply
choose $s_0$ with $\om(s_0)=\bar\om$
and then, with an arbitrary $I_0\in\Om$, take $s=\Upsilon(s_0,I_0)$ and
define $T_s:=T(\cdot,I_0,s_0)$. So to complete the proof in the present
case, we just need  find an uncountable
subset $W$ of $V_\ka$ such no two vectors of $W$ are linearly
dependent. Such a set exists because, as $V_\ka$ has positive measure,
there are uncountably many lines through the origin that intersect
$V_\ka$. Thus, we can pick a unique vector from $V_\ka$ in any such
line to form the set $W$.

 Next, still assuming that $\nabla \om(s)$ has
 rank $n$, we remove the analyticity assumption:
 $ \om(s)$ is now of class $ C^k$. 
 We  make, without loss of generality,
 a simplifying assumption that $d=n$ and $\om$ is a
 diffeomorphism of $B$ onto its image $V$. This can always be achieved
 by replacing $B$ by a small neighborhood of some arbitrarily
 fixed $s^0\in B$  and dropping some ``disposable'' parameters. More precisely,
 relabeling the parameters $s_1,\dots,s_d$,
 we may assume that the matrix  
 \begin{equation*}
   \left[\ \partial_{s_1}\om(s)\dots \partial_{s_n}\om(s) \ \right]
 \end{equation*}
 has rank $n$ for all $s\approx s^0$. Then, if $d>n$,
 we consider only those $s\in B$ whose last $d-n$
 components, $s_{n+1},\dots,s_d$, are fixed and equal to the
 last $d-n$ components of $s^0$. Accordingly, we replace $B$ by a neighborhood
 $\tilde B$ of $s^0$ in the corresponding $n$-dimensional affine space. 
 With the number of parameters equal to $n$,
 the rank condition implies that $\om$ is a diffeomorphism, possibly
 after the neighborhood  $\tilde B$ of $s^0$ is made smaller. Of course, 
 proving the statement of the lemma with $B$ replaced by the smaller
 set $\tilde B$ trivially implies the original statement.

 The assumption that $\om:B\to V$ is a diffeomorphism
 allows us to reparameterize the problem, using the frequency vectors
 as parameters, in such a way that the linear 
 integrable part becomes analytic in the parameters. For that we
 denote by  $\upsilon:V\to B$ the inverse to $\om(s)$; this is a
 $ C^k$ map. Let again $\ka>0$ be so small that $V_\ka$ has
 positive measure.  
 Clearly,  Theorem \ref{thm:BHT} applies to
 the integrable Hamiltonian $H^0(I,\bar\om):=\bar \om \cdot I$, $I\in
 \Om$, $\bar\om\in V$, and the perturbation 
 $H^1(\theta,I,\bar\om):=G^1(\theta,I,\upsilon(\bar\om))$, provided
  $G^1:\torus^n\times\Omega\times B\to\R$ has sufficiently small
  $ C^k$-norm. This implies the conclusion of
  Lemma \ref{le:NDs} (we choose a subset $W\subset V_\ka$ with the required
  properties as in the first part of the proof). 
  Thus Lemma \ref{le:NDs} is proved in the case that $\nabla \om(s)$
  has rank $n$.

  Finally, we take on the case of the rank of $\nabla \om(s)$ being
  less than $n$; by (NDs), the rank has to be equal to $n-1$,
  with the vector
  $\om(s)$ outside the range of  $\nabla \om(s)$ for each $s\in B$.
  We introduce an extra real parameter $\be\approx 1$, so the
  parameter set becomes $B\times (1-\ep,1+\ep)$ for a small
  $\ep>0$. Consider the linear integrable Hamiltonian
  $\tilde G^0(I;s,\be):=\be \om(s)\cdot I$ and  the perturbation 
  $\tilde G^1(I,\theta;s,\be):=\be G^1(\theta,I,s)$.
  Due to (NDs), the gradient matrix
  \begin{equation*}
    \nabla_{s,\be}(\be\om(s))= \left[\ \be\nabla_{s} \om(s)\quad \om(s) \ \right]
  \end{equation*}
  has rank $n$ for all $(s,\ep)\in B\times
  (1-\ep,1+\ep)$ if $\ep>0$ is small enough, which we will henceforth
  assume. 

  Thus, the part of the statement of Lemma \ref{le:NDs} already
  proved above applies to $\tilde G^0$, $\tilde G^1$, provided
  $G^1:\torus^n\times\Omega\times B\to\R$ has sufficiently small
  $ C^k$-norm. This
  yields a set  $\tilde W\subset \R^n$
 consisting of rationally independent vectors, no two of them
 being linearly dependent, such that for every 
 $\bar \om\in \tilde W$ the Hamiltonian system
\begin{equation} \label{eq:17b}
  \begin{aligned}
     \dot \theta&= \be \nabla_I G(\theta,I;s),\\  
     \dot I&=-\be \nabla_\theta G(\theta,I;s),
  \end{aligned}
\end{equation}
has for some $s\in B$, $\be\in (1-\ep,1+\ep)$ a  quasiperiodic
solution with frequency vector $\bar \om$. Noting that \eqref{eq:17b} is just
\eqref{eq:17} with rescaled time, we get the desired conclusion for
\eqref{eq:17} with a set $W$ obtained from $\tilde W$ by multiplying
each element $\bar\om\in \tilde W$ by a scalar
$\be=\be(\bar\om)\approx 1$. The vectors obtained this way are
mutually distinct, due to 
the properties of $\tilde W$, so $W$ is still uncountable, and the
pairwise linear independence is obviously preserved as well.
The lemma is proved. 
\end{proof}

We remark that for the matrix $\nabla \om(s)$ to have rank $n$, we
would need  $d\ge n$. Hypothesis
(NDs), on the other hand, only requires $d\ge n-1$, which ``saves'' us
one parameter.

We are now ready to complete  the proof of Theorem
\ref{thm:new}.
\begin{proof}[Proof of Theorem \ref{thm:new}]
  We return to the Hamiltonian of the reduced equation
  (see \eqref{eq:11} and \eqref{eq:PhiThetaI}). In the coordinates
  $(\xi,\eta)$,
\begin{equation}
  \label{eq:11a}
  \Phi(\xi,\eta;s):=\frac{1}{2}\sum_{j=1}^n
\omega_j(s)(\xi_j^2+\eta_j^2)+\hat \Phi(\xi,\eta;s),
\end{equation}
and in the action-angle variables $J=(J_1,\dots,J_n)\in \R^n$,
  $\theta=(\theta_1,\dots,\theta_n)\in\torus^n$ (cp.~\eqref{eq:12}),
  \begin{equation}
    \label{eq:13}
  \Phi(\theta,J;s)=\omega(s)\cdot J+\hat\Phi(\theta,J;s).
\end{equation}
Here, $J$ is taken near the origin and such that
$J_j>0$ for all $j\in\{1,\dots,n\}$,
and $s\in B_\de\subset \R^d$, for some $\de>0$.  

Recall that $\hat \Phi(\xi,\eta;s)$ is of class $C^K$ on a
neighborhood of the origin in $\R^{2n}\times \R^d$ 
and of order $\order(|(\xi,\eta)|^3)$ as $(\xi,\eta)\to (0,0)$.
Therefore, by Taylor's theorem, $\hat \Phi(\xi,\eta;s)$
can be written as the sum of finitely many terms, each of them being
the product of a degree-three monomial in $\xi,\eta$
and a $ C^{K-3}$ function of $\xi$, $\eta$, $s$.
The function $\hat \Phi(\theta,J;s)$ is obtained
from this sum by substituting
$$(\xi_j,\eta_j)=\sqrt{2J_j}(\cos\theta_j,\sin\theta_j)\quad(j=1,\dots,n)$$
(which introduces some singular behavior in the derivatives of
$\hat \Phi(\theta,J;s)$ as $J\to 0$). In these action-angle variables,
$\hat\Phi$ is of order $\order(|J|^{3/2})$ as $|J|\to 0$.

Recall also that $\om(s)\in \R^n$ is as in \eqref{eq:10} and it is of
class $C^{K+1}$ as a function of $s$. 

Fix constants $k_0\le K-3$ and $\nu>n-1$, $k_0$ being an integer,
such that $k_0>4\nu +2$.
This is possible due to \eqref{eq:Kn}. According to Remark \ref{rm:BHT}(iii), 
Theorem \ref{thm:BHT} applies with these choices of $\nu$ and $k_0$.
We introduce the scaling $J=\ep I$ with $\ep \in (0,1)$, $I\in
\Om$, where
  \begin{equation}
    \label{eq:30}
\Om:=\{I\in\R^n:\ q\leq I_j\leq 2q\ \,(j=1,\dots,n)\}
\end{equation}
and $q$ is some positive constant, which we fix for the rest of the proof.
Now define $G^0$, $G^1$ on  $ \torus^n\times\Om \times B_\de$ by
\begin{equation}
  \label{eq:16}
  \begin{aligned}
  G^0(I;s)&:=\om(s)\cdot I, \\
  G^1(\theta,I;s)&:=\frac1\ep\, \hat
  \Phi(\theta,\ep I;s),
\end{aligned}
\end{equation}
which is legitimate for all sufficiently small  $\ep>0$ (below we will
make an additional smallness requirement on $\ep$). We set $G:=G^0+G^1$.

Observe that  
$G(\theta, I;s)=\Phi(\theta,\ep I;s)/\ep$,
which is the right Hamiltonian for
the rescaled reduced equation \eqref{eq:28}: the Hamiltonian
system corresponding to the Hamiltonian $G$ in the standard symplectic
form is the same as the system obtained from \eqref{eq:28} after the
substitution $J=\ep I$ (and it is of course the same as
the Hamiltonian system of  $\Phi$ with respect to the transformed
symplectic form corresponding to the 
\emph{noncanonical} coordinate transformation $(I,\theta)=(\ep
J,\theta)$).  

We are now going to apply Lemma \ref{le:NDs} to the Hamiltonian
$G=G^0+G^1$, with $\ep>0$ sufficiently small. Take $k:=K-3\ge k_0$. 
The smoothness hypotheses of Lemma \ref{le:NDs}
on $s\to\om(s)$ and $G^1$ are then satisfied.
Hypothesis (NDs)  is verified, possibly after $\de>0$ is
made smaller, due to hypothesis  (ND) in Section \ref{sec:mainres}.
It remains to verify that the smallness requirement on $G^1$ is
met if $\ep>0$ is small enough. Consider any derivative
$D^\al\hat\Phi(\theta,J;s)$ of order at most $k$. Here $\al$ is a
multiindex in $\N^{2n+d}$. We denote by $\al_J$ the total number
of derivatives in $D^\al\hat\Phi$ taken with respect to the
$J$-variables. Using our previous observations on the asymptotic
behavior of $\hat 
\Phi$ as $J\to 0$ and taking into account the maximal singularity possibly
introduced by differentiating one of the roots $J_1^{1/2},
\dots,J_n^{1/2}$, we obtain that $D^\al\hat\Phi(\theta,J;s)$ is of
order $|J|^{3/2-\al_J}$ as $|J|\to 0$. Therefore, taking the corresponding
derivative $D^\al$ in the variables $(\theta,I;s)$, we discover that
for some constant $C_\al$
\begin{equation*}
   |D_{\theta,I;s}^\al G^1(\theta,I;s)|=\frac1\ep\,
   \ep^{\al_J}|D^\al_{\theta,J;s} \hat 
 \Phi(\theta,\ep I;s)|\le C_\al \ep^{1/2}\quad
  ((\theta,I,s)\in\torus^n\times \Om\times B_\de).
\end{equation*}
This implies that if $\ep>0$ is sufficiently small, the condition
$\|G^1\|_{ C^k(\torus^n\times\Omega\times B)}<\vartheta$ of Lemma
\ref{le:NDs} is satisfied. 

Having verified all the hypotheses, and
fixing a small enough $\ep>0$,
we obtain  that 
the system \eqref{eq:17} has quasiperiodic solutions with frequencies
covering the set $W$, as stated in Lemma \ref{le:NDs}. The
trajectories of these solutions are contained in $\torus^n\times\Om$.
Undoing the $\ep$-scaling, we obtain quasiperiodic solutions of 
the reduced equation \eqref{eq:28} whose trajectories are contained in  
$\torus^n\times\ep\Om$. If so desired, we can adjust $\ep>0$
to guarantee that the trajectories are contained
in any given neighborhood of $\torus^n\times\{0\}$.

We now reverse the transformations made in  Section \ref{sec:reduced},
namely, the passage to the action-angle variables, transformation
\eqref{eq:symplin}, and the Darboux transformation, 
to get back to the reduced equation  \eqref{eq:reduced2}.
This yields quasiperiodic solutions of 
\eqref{eq:reduced2}, for the same values of $s$ as in
\eqref{eq:17}, whose frequencies vectors cover the same set $W$.
Moreover,
we can assume that the trajectories 
of these solutions  
are all contained in a small neighborhood of the origin
(we may need to adjust $\ep>0$ for this, as noted above). 
In particular, if  $z$ is any of these solutions, then 
 $z(y)\in \mathscr{N}$ for all $y\in\R$, 
$\mathscr{N}$ being the neighborhood of $0\in Z$ from Proposition
\ref{prop:wc}. Then, 
by Proposition \ref{prop:wc}(b), 
\[
U(y)=(U_1(y),U_2(y))^T=z(y)+\sigma(\{z(y)\}_\basis;s)\in Z
\]
is a  solution of system \eqref{eq:abstractsys}. Letting 
\begin{equation}
  \label{eq:34}
  u(x,y)=U_1(y)(x),
\end{equation}
we obtain a solution of
\eqref{eq:a1}. This solution is quasiperiodic in $y$,
$2\pi$-periodic and even in $x_N$, and radially symmetric in $x'$ (the
periodicity and symmetry come from the definition of the space $Z$). The
frequencies of the solutions obtained this way still
cover the same set $W$, which has the properties required in
Theorem \ref{thm:new}. 
It remains to show that each solution $u(x,y)$ obtained this way 
decays to 0 as $|x'|\to\infty$, uniformly in $x_N$ and $y$.
This is a direct consequence of the
fact that the set $\{u(\cdot,y):y\in \R\}$ is contained in a compact
set---continuous image of a torus---in  $H_\rade^{m+2}(\RS)$,
with $m >N/2$. 
\end{proof}

\begin{rmkn}
  \label{rm:discussion}
  As noted at the end of Section \ref{sec:reduced}, the reduced equation
  is reversible and this structure can be used instead of the
  Hamiltonian structure in the proof of Theorem \ref{thm:new}.
  Theorems for reversible systems analogous to Theorem \ref{thm:BHT}
  can be found in \cite{Broer-H-N, Broer-H:reversible,
    Sevryuk:KAM-stable}, for example, 
and a result analogous to our Lemma \ref{le:NDs}
can be derived from those. For analytic reversible systems,
Scheurle has proved the existence of quasiperiodic solutions
under the same nondegeneracy condition as (NDs), see
\cite{Scheurle:ode}.  
\end{rmkn}

\section{Proof of Theorem \ref{thm:main}} \label{proof2}

Assume the hypotheses of Theorem \ref{thm:main} to be satisfied. 
We derive the conclusion of  the theorem from Theorem
\ref{thm:new} with $n:=2$, $K:=10>4n+1$, $m:=\ell-14>N/2$, with $\ell$
as in  hypothesis (S). 
Note that $f$ is of class $ C^{K+m+4}$.

To put equation \eqref{eq:1} in the form \eqref{eq:a1},
we linearize a rescaled equation \eqref{eq:1} about a ground state.
Here we initially follow \cite{Dancer:new}.
Let $\varphi$ be a (radially symmetric) ground state of
\eqref{eq:gsb}, as in hypothesis (G).
As assumed in (G), the operator $-\De-f'(\varphi(x'))$ considered on
$L^2_\rad(\R^{N-1})$ with domain $H^2_\rad(\R^{N-1})$ has exactly one
nonpositive eigenvalue,  further denoted by $\mu_0$, and this
eigenvalue is negative and simple.  
For $\la>0$ set
$\varphi^\lambda(x'):=\varphi(\sqrt{\lambda} x')$. This is a ground
state of the rescaled equation
\begin{equation}\label{eq:gsr}
\Delta u+\la f(u)=0,\quad x'\in\R^{N-1}.
\end{equation}

In the following, we view $\varphi^\la$ as a function of $x\in\R^N$,
independent of $x_N$. Set
$$a^\la(x):=\la f'(\varphi^\la(x)).$$ 
We examine the Schr\"{o}dinger operator 
$A^\la:=-\Delta-a^\la(x)$ acting on $L^2_\rade(\RS)$ 
with domain $H^2_\rade(\RS)$. The function $a^\la$
has the limit $\la f'(0)$ as $|x'|\to\infty$,
which is negative due to hypothesis (S). As noted in Section
\ref{sec:mainres}, this implies that the essential
spectrum of $A^\la$ is contained in $[-\la f'(0),\infty)$. 
 Scaling and separation of variables show, as in
 \cite{Dancer:new}, that the following statements hold. The
 principal (minimal) eigenvalue of $A^\la$ is $\la\mu_0<0$
 with eigenfunction independent of $x_N$, and it is a simple
 eigenvalue. 
 If $\la$ is greater than but  close to $-1/\mu_0>0$,
 then the second eigenvalue is $\la\mu_0+1<0$ with eigenfunction of the
 form $\varsigma(|x'|)\cos x_N$ and it is also a simple eigenvalue.
 All  other eigenvalues (as well as the essential spectrum) of $A^\la$
 are positive.
 Fix any $\la> -1/\mu_0$, $\la\approx -1/\mu_0$, with these properties
 and set
 \begin{align}
   \label{eq:4}
   a(x;s)&:=a^{\la+s}(x)=(\la+s)f'(\varphi^{\la+s}(x)),\\
   f_1(x,u;s)&:=(\la+s)f(\varphi^{\la+s}(x)+u)-a(x;s)u. \label{eq:4b}
 \end{align}
 Here  $s\in (-\de,\de)=:B$, where we take $\de\in (0,\la)$ so small
 that for all
 $s\in [-\de,\de]$ 
 \begin{equation}
   \label{eq:5}
   \mu_1(s):=(\la+s)\mu_0<\mu_2(s):=(\la+s)\mu_0+1<0
 \end{equation}
and  $\mu_1(s)$, $\mu_2(s)$ are the only nonpositive eigenvalues of
  $-\De-a(x;s)$. Thus, the function $a(x;s)$ satisfies hypotheses
  (A1)(a) (with $L:=(\la-\de)f'(0)$) and (A2)(b) (with $n=2$).

  Obviously,
  $f_1$ satisfies \eqref{eq:f}, and the symmetry requirements in (S1),
  (S2) follow from the definitions of $a$, $f_1$, and the symmetry of
  $\varphi^{\la+s}(x')= \varphi(x'(\la+s)^{1/2})$. 
The verification of the smoothness requirements in (S1), (S2), with
$d=1$, is
straightforward (and is left to the reader)
when one uses the following claim: $\varphi$ is of
class $ C^{K+m+5}$ and all its derivatives up to order $K+m+5$
decay exponentially as $|x'|\to\infty$.
To prove this claim, we first note that, since $f$ is of class
$ C^{K+m+4}$, the fact that   $\varphi$ is of
class $ C^{K+m+5}$ (with locally H\"older derivatives of order
$K+m+5$) is a standard elliptic regularity result. Now, since
$\varphi(x')$---and consequently $f(\varphi(x'))$---
decays exponentially, the equation $\De \varphi(x')=-f(\varphi(x'))$
and local elliptic estimates \cite{Gilbarg-T} imply that the same is
true for the first order derivatives of $\varphi$. Differentiating the
equation and iterating the estimates a finite number of times,
one eventually obtains that all 
derivatives of $\varphi$  up to order $K+m+5$ decay exponentially,
proving the claim.

Finally, to verify hypothesis (ND) with $n=2$, we
take
\begin{equation*}
  \omega_1(s):=\sqrt{(\la+s)|\mu_0|}, \quad
  \omega_2(s):=\sqrt{(\la+s)|\mu_0|+1},
\end{equation*}
$\om(s):=(\om_1(s),\om_2(s))^T$,
and compute the determinant of the $2\times 2$ matrix
$\big[\,\om'(0)\ \  \om(0)\, \big]$:
\begin{equation*}
  \begin{aligned}
  \det \big[\,\om'(0)\ \  \om(0)\, \big]&=\frac{|\mu_0|}{2}
  \left(\frac{\sqrt{\la|\mu_0|+1}}{\sqrt{\la|\mu_0|}}-\frac{\sqrt{\la|\mu_0|}}{
   \sqrt{\la|\mu_0|+1}} \right)\\&=\frac{|\mu_0|}{2}
  \frac{1}{\sqrt{\la|\mu_0|(\la|\mu_0|+1)}}\ne 0.
  \end{aligned}
\end{equation*}
Hence, (ND) holds as well and we may now apply Theorem
\ref{thm:new} with $n=2$. 

Let $W\subset \R^2$ be as in the conclusion of Theorem
\ref{thm:new}. Thus for any  $\bar \om \in W$ there exist
$s\in (-\de,\de)$ and a solution  $v(x,y)$ of the equation 
\begin{equation*}
\Delta v+v_{yy}+a(x;s)v+
f_1(x,v;s)=0\quad (x\in \R^N,\ y\in\R),
\end{equation*}
such that \eqref{eq:decper} holds with $u$ replaced by $v$,
and $v(x,y)$ is radially symmetric
in $x'$, even  and $2\pi$-periodic in $x_N$, 
and quasiperiodic in $y$ with the frequency vector $\bar\om$.
By the definition of $a$ and $f_1$, 
$\tilde u=\varphi^{\la+s}+v$ is a solution of
\begin{equation*}
\Delta \tilde u+\tilde u_{yy}+(\la+s)f(\tilde u)=0\quad (x\in \R^N,\ y\in\R),
\end{equation*}
with the same properties as $v$.
Using the rescaling $u(x,y)=\tilde u(x(\la+s)^{-1/2},
y(\la+s)^{-1/2})$ we obtain a solution of the original equation
\eqref{eq:1} which satisfies \eqref{eq:decper}, and is radially symmetric
in $x'$, even  and $2\pi(\la+s)$-periodic in $x_N$, 
and quasiperiodic in $y$ with the frequency vector $(\la+s)\bar\om$
(obviously, any such vector is nonresonant, just as $\bar\om$). Since
no two vectors in (the uncountable set) $W$ are linearly dependent, the set of
frequency vectors  obtained this way is uncountable. 
So we have a family of solutions of \eqref{eq:1} with the desired
properties, we just need verify that they are all positive. 
This follows from  \eqref{eq:extendf}. Indeed, let $u$ be any of these
solutions. Since it is  quasiperiodic (in
the sense of our definition), it is not periodic in $y$ and in
particular $u\not \equiv 0$. By the strong maximum principle, either $u>0$
or $u$ is negative somewhere. In the latter case, 
quasiperiodicity and 
\eqref{eq:decper} imply that $u$ 
has a local negative minimum at some point.  But
at that point equation \eqref{eq:1} cannot be satisfied
when  \eqref{eq:extendf} holds. Thus  $u>0$.

The proof of Theorem \ref{thm:main} is now complete.

\bibliographystyle{amsplain} 
\providecommand{\bysame}{\leavevmode\hbox to3em{\hrulefill}\thinspace}
\providecommand{\MR}{\relax\ifhmode\unskip\space\fi MR }
\providecommand{\MRhref}[2]{%
  \href{http://www.ams.org/mathscinet-getitem?mr=#1}{#2}
}
\providecommand{\href}[2]{#2}

\end{document}